\documentclass{article}

\usepackage{amsmath,amsfonts,latexsym,amssymb,amsthm,url}

\usepackage[utf8]{inputenc}

\DeclareMathOperator{\charac}{char}
\DeclareMathOperator{\Gal}{Gal}

\newcommand{\C}{{\mathbb C}}

\newcommand{\Q}{{\mathbb Q}}
\newcommand{\R}{{\mathbb R}}
\newcommand{\Z}{{\mathbb Z}}

\newcommand{\gerp}{\mathfrak{p}}
\newcommand{\gerP}{\mathfrak{P}}

\newcommand\ab{{\mathrm{ab}}}
\newcommand\tho{{\text{th}}}

\newcommand{\norm}{\mathcal{N}}
\newcommand{\DD}{\mathcal{D}}
\newcommand{\calS}{\mathcal{S}}

\newtheorem{theorem}{Theorem}[section]

\newtheorem{proposition}[theorem]{Proposition}

\newtheorem{corollary}[theorem]{Corollary}

\newtheorem{lemma}[theorem]{Lemma}

\newtheorem{remark}[theorem]{Remark}

\numberwithin{equation}{section}

\title{The Chevalley-Bass Theorem}

\author{Yuri Bilu\footnote{Institut de Mathématiques de Bordeaux, Université de Bordeaux \& CNRS; \textsf{yuri@math.u-bordeaux.fr}}}

\date{version of \today}

\makeatletter

\renewcommand*\l@section[2]{
	\ifnum \c@tocdepth >\z@
	\addpenalty\@secpenalty
	\addvspace{0.2em \@plus\p@}
	\setlength\@tempdima{1.5em}
	\begingroup
	\parindent \z@ \rightskip \@pnumwidth
	\parfillskip -\@pnumwidth
	\leavevmode \bfseries
	\advance\leftskip\@tempdima
	\hskip -\leftskip
	#1\nobreak\hfil \nobreak\hb@xt@\@pnumwidth{\hss #2}\par
	\endgroup
	\fi}

\makeatother

\begin{document}

\hfuzz=4pt

	\maketitle

\begin{abstract}
This is an exposition of a theorem due to Chevalley (1951) and Bass (1965). Let~$K$ be a finitely generated field of characteristic~$0$. Then there exists a positive integer~$\Lambda$, depending only on~$K$, such that, for every positive integer~$n$, the following holds: if ${\alpha\in K}$ is a $\Lambda n^\tho$ power in the cyclotomic extension $K(\zeta_{\Lambda n})$, then~$\alpha$ is an $n^\tho$  power in~$K$. 

We also give explicit expressions for a suitable~$\Lambda$ of two kinds: one in terms of the degree of the maximal abelian subfield~$K_\ab$ of~$K$, the other in terms of the discriminant of $K_\ab$. 
\end{abstract}

	{\footnotesize
		
		\tableofcontents
		
	}

\section{Introduction}

Let~$K$ be a number field and~$n$ a positive integer. Many Diophantine problems depend on the following question: assume that ${\alpha\in K}$ is an $n^\tho$ power in the cyclotomic extension $K(\zeta_n)$; is it true that it is  an $n^{\text{th}}$ power in~$K$? In symbols: is it true that ${K(\zeta_n)^n\cap K=K^n}$? 

This is wrong in general: ${-4=(1+i)^4}$ is a $4^{\text{th}}$ power in $\Q(i)$ but not in~$\Q$. However, a slightly weaker statement holds true. For a field~$K$ of characteristic~$0$ we denote by $K_\ab$ the maximal abelian subfield of~$K$; that is, the biggest subfield of~$K$ which is an abelian extension of~$\Q$.

\begin{theorem}
\label{thbass}
[Chevalley, Bass]
Let~$K$ be a field of characteristic~$0$ with the property 
\begin{equation}
\label{ehypab}
\text{$K_\ab$ is a finite extension of~$\Q$}. 
\end{equation}
Then there exists a positive integer~$\Lambda$ such that for every positive integer~$n$ the following holds: if ${\alpha\in K}$ is a $\Lambda n^\tho$ power in  $K(\zeta_{\Lambda n})$, then~$\alpha$ is an $n^\tho$  power in~$K$.   In symbols: ${K(\zeta_{\Lambda n})^{\Lambda n}\cap K\subset K^n}$. 
\end{theorem}

Note that hypothesis ${[K_\ab:\Q]<\infty}$ holds trivially when~$K$ is a number field. More generally, it holds when~$K$ is a finitely generated field, see Exercise~4 in \cite[Chapter~VIII]{La02}.

%The present note is an exposition of the proof of this  beautiful theorem.

The case of prime power~$n$  is due to Chevalley~\cite{Ch51}, and does not require hypothesis~\eqref{ehypab}. Moreover, if we restrict to odd prime power~$n$, then we may take ${\Lambda=1}$, see Section~\ref{sch} below. The same applies when~$n$ is a power of~$2$ under the additional hypothesis ${i\in K}$.   

Bass~\cite{Ba65} extended the work of Chevalley to arbitrary~$n$; his stated result is formally weaker than Theorem~\ref{thbass}, but what he proves is exactly Theorem~\ref{thbass}. Smith~\cite{Sm70} gave a very succinct proof of the Theorem of Bass and corrected some inaccuracies. In Sections~\ref{skey} and~\ref{sbass} we follow Smith's argument, adding some details. 

Note that both Bass and Smith assumed~$K$ a number field. We extended their argument to the more general case of a field satisfying~\eqref{ehypab}, using a suggestion of Georges Gras~\cite{Gr23}. 

This article is purely expository, no really new ideas are involved. The results of Section~\ref{schba} are, formally, new, but they are deduced from the results of the previous sections, using very standard arguments.

\paragraph{Acknowledgments} I am indebted to Keith Conrad for introducing me to the topic, and showing me the articles of Chevalley, Bass and Smith. I am most grateful to Georges Gras for inspiring comments on an early version of this note, and many illuminating discussions.  I thank Denis Benois for  helpful suggestions, and Peter Müller for showing me  
articles~\cite{AG89} and~\cite{KP89}.  

The computations in Section~\ref{schba} were performed using \textsf{PARI}~\cite{pari}. I thank the \textsf{PARI}  team, and Bill Allombert in particular, for their availability. 

I worked on this note during my stay at MPIM Bonn in June 2023; I thank this institute for financial support and stimulating working conditions. 

This work was partially supported by the ANR project JINVARIANT.

\paragraph{Notation and conventions}
In this note~$K$ is a field %of characteristic~$0$  
and~$\bar K$ is a fixed algebraic closure of~$K$. Given ${\alpha \in K}$ and a positive integer~$n$, we denote by $\alpha K^n$ the set of $n^\tho$ powers in~$K$ multiplied by~$\alpha$:
$$
\alpha K^n:= \{\alpha\beta^n: \beta \in K\}. 
$$
%Since ${\Q\subset K}$, this also fixes an algebraic closure~$\bar\Q$. 
For a positive integer~$n$, not divisible by the characteristic of~$K$,  we let ${\zeta_n\in \bar K}$ be a primitive root of unity of order~$n$. %In particular, the field ${K\cap\Q(\zeta_p)}$ in~\eqref{ehyproots} is well-defined.  
We denote by~$\mu_n$ the group of roots of unity of order~$n$, viewed as a subgroup of $\bar K^\times$ generated by~$\zeta_n$. We use the (slightly abusive) notation ${i=\sqrt{-1}=\zeta_4}$. 

Everywhere throughout the text~$p$ denotes a prime number, and $m,n,r,s$ denote strictly positive integers. 

\section{An Irreducibility Theorem}

Recall the following well-known irreducibility criterion for  binomials; see,  for instance, Theorem~9.1  in~\cite[Chapter~VI]{La02}. 

\begin{theorem}
\label{thlang}
Let~$K$ be a field, ${\alpha\in K^\times}$ and~$n$  a positive integer. Assume 
that for all  ${p\mid n}$ we have ${\alpha\notin K^p}$. If ${4\mid n}$ then we assume, in addition, that ${\alpha\notin -4K^4}$. Then the polynomial ${X^n - \alpha}$ is irreducible in $K[X]$. 
\end{theorem}

We will use this theorem through the following consequences.

\begin{corollary}
\label{cps}
Let~$K$ be a field and~$p$ a prime number. If ${p=2}$ and ${\charac K\ne 2}$ then assume that ${i\in K}$. Let ${\alpha \in K}$, but ${\alpha \notin K^p}$. Then for any positive integer~$s$ the polynomial ${X^{p^s}-\alpha}$ is irreducible in $K[X]$. 
\end{corollary}

\begin{corollary}
\label{crootgal}
Let~$K$ be a field and~$p$ a prime number distinct from the characteristic of~$K$. Let ${r\ge 2}$ be such that ${\zeta_{p^{r-1}}\in K}$, but ${\zeta_{p^{r}}\notin K}$. Then $K(\zeta_{p^r})$ is a Galois extension of~$K$ of degree~$p$. 
\end{corollary}

Both corollaries are straightforward consequences of Theorem~\ref{thlang}, though, perhaps, 
 the case ${p=2}$ and ${\charac K\ne 2}$ of Corollary~\ref{cps} needs some explanation.    In this case ${-4=(1+i)^4\in K^4}$. Hence ${\alpha \notin K^2}$ implies ${\alpha\notin -4K^4}$, and Theorem~\ref{thlang} applies.

\section{The Prime Power Case}
\label{sch}
In this section we prove the Theorem of Chevalley.

\begin{theorem}[Chevalley]
\label{thch}
Let~$K$ be a field and~$p$ a prime number distinct from the characteristic of~$K$. 
Let~$n$ be a power of~$p$. Assume that ${i\in K}$ if ${p=2}$. Let ${\alpha\in K^\times}$ be such that ${\alpha \in K(\zeta_n)^n}$. Then  ${\alpha\in K^n}$. In symbols:
${K(\zeta_n)^n\cap K=K^n}$. 
\end{theorem}

Writing ${n=p^r}$, we need to prove that ${\alpha \in K(\zeta_{p^r})^{p^r}}$ implies ${\alpha\in K^{p^r}}$. It will be more convenient to separate the field and the exponent, proving the following more general statement. 
\begin{theorem}
\label{thchbis}
Let~$K$ and~$p$ be as in Theorem~\ref{thch}; in particular, we assume that ${i\in K}$ if ${p=2}$.
Let ${s\ge r\ge 1}$; then ${\alpha \in K(\zeta_{p^r})^{p^s}}$ implies that ${\alpha\in K^{p^s}}$. In symbols:
${K(\zeta_{p^r})^{p^s}\cap K=K^{p^s}}$. 
\end{theorem}

\begin{proof} 
We use induction in~$r$. 
In the case ${r=1}$  the proof is very easy. Write ${\alpha=\beta^{p^s}}$ for some ${\beta\in K(\zeta_p)}$. Taking the norms, we obtain ${\alpha^d=\gamma^{p^s}}$, where ${d=[K(\zeta_p):K]}$ and ${\gamma=\norm_{K(\zeta_p)/K}\beta}$. 
Since ${d\le p-1}$, we have ${\gcd(p^s,  d)=1}$. Let ${u,v\in \Z}$ be such that ${ud+vp^s=1}$. Then
${\alpha=\alpha^{ud+vp^s}=(\gamma^u\alpha^v)^{p^s}\in K^{p^s}}, $
as wanted.

Now assume that ${r\ge 2}$. If ${p=2}$ then ${\zeta_{4}\in K}$ by the hypothesis. This means that, in the case ${p^r=4}$, there is nothing to prove. Thus, when ${p=2}$, we may assume that ${r\ge 3}$.

We are assuming that ${\alpha \in K(\zeta_{p^r})^{p^s}}$, and want to prove that 
\begin{equation}
\label{ealin}
\alpha \in K(\zeta_{p^{r-1}})^{p^s}.
\end{equation}
By induction, it would then follow that ${\alpha\in K^{p^s}}$, as wanted. 

If ${\zeta_{p^r}\in  K(\zeta_{p^{r-1}})}$ then there is nothing to prove. Hence we assume that ${\zeta_{p^r}\notin  K(\zeta_{p^{r-1}})}$. Corollary~\ref{crootgal} implies now that $K(\zeta_{p^{r}})$ is a Galois extension of $K(\zeta_{p^{r-1}})$ of degree~$p$. 

The field $K(\zeta_{p^{s}})$ is also a Galois extension of $K(\zeta_{p^{r-1}})$. Let~$\sigma$ be an element in  $\Gal\bigl(K(\zeta_{p^{s}})/K(\zeta_{p^{r-1}})\bigr)$ such that  the restriction  $\sigma\vert_{K(\zeta_{p^{r}})}$ generates the Galois group of $K(\zeta_{p^{r}})$ over $K(\zeta_{p^{r-1}})$. We have ${\sigma(\zeta_{p^s})=\zeta_{p^s}^g}$ for some ${g\in \Z}$, not divisible by~$p$. Since ${\sigma(\zeta_{p^{r-1}})=\zeta_{p^{r-1}}}$, we must have ${g\equiv 1\pmod {p^{r-1}}}$. In particular,
\begin{equation}
\label{econgg}
g\equiv 
\begin{cases}
1\pmod {p}, &\text{if $p\ge 3$},\\
1\pmod 4,&\text{if ${p=2}$}.
\end{cases} 
\end{equation}

Write ${\alpha=\beta^{p^s}}$ for some ${\beta \in K(\zeta_{p^r})}$. We make the following two observations.

\begin{enumerate}
\item
Since ${\beta^{p^s}\in K \subset K(\zeta_{p^{r-1}})}$, we have 
\begin{equation}
\label{esigbb}
\sigma(\beta)/\beta \in\mu_{p^s}.  
\end{equation}
We write ${\sigma(\beta) =\beta\zeta_{p^s}^m}$, where ${m\in \Z}$. 

\item
Since ${\beta \in K(\zeta_{p^r})}$, and the restriction $\sigma\vert_{K(\zeta_{p^{r}})}$ is of order~$p$ in the group $\Gal\bigl(K(\zeta_{p^r})/K(\zeta_{p^{r-1}})\bigr)$, we have ${\sigma^p(\beta) =\beta}$. In terms of the integers~$m$ and~$g$ introduced above, this reads as 
\begin{equation*}
m(1+g+\cdots +g^{p-1}) \equiv 0 \pmod{p^s}. 
\end{equation*}
\end{enumerate}

Now we are ready to complete the proof. Using~\eqref{econgg}, a standard argument implies that 
$$
1+g+\cdots +g^{p-1}= \frac{g^p-1}{g-1}\equiv p\pmod{p^2}. 
$$
Hence ${m\equiv 0\pmod{p^{s-1}}}$, which means that~\eqref{esigbb} can be drastically refined:
${\sigma(\beta)/\beta \in\mu_{p}}$. 

On the other hand, since ${\zeta_{p^r}^p\in K(\zeta_{p^{r-1}})}$, we also have ${\sigma(\zeta_{p^r})/\zeta_{p^r}\in \mu_p}$. Moreover, since ${\zeta_{p^r}\notin K(\zeta_{p^{r-1}})}$, the $p^\tho$ root of unity $\sigma(\zeta_{p^r})/\zeta_{p^r}$ must be primitive. It follows that there exists ${b\in \Z}$ such that 
${\sigma(\beta)/\beta=(\sigma(\zeta_{p^r})/\zeta_{p^r})^b}$, which can be re-written as ${\sigma(\beta')=\beta'}$, where ${\beta':=\beta\zeta_{p^r}^{-b}}$. 

Thus, ${\beta' \in K(\zeta_{p^{r-1}})}$. Since ${(\beta')^{p^s}=\alpha}$, this implies~\eqref{ealin}. The theorem is proved.
\end{proof}

Hypothesis ``~${i\in K}$ if ${p=2}$~'' cannot be dropped. For instance, take ${K=\R}$. Then ${-1}$  not a square in~$K$, but for every~$m$ it is a $4m^\tho$ power in ${K(\zeta_{4m})=\C}$. However, Chevalley made the following important observation.

\begin{proposition}[Chevalley]
\label{pnoti}
Let~$K$ be a field with ${\charac K\ne 2}$ and~$k$ a positive integer such that %${\zeta_{2^k}\in K(i)}$, but 
%\begin{equation*}\label{enotin}
${\zeta_{2^{k+1}}\notin K(i)}$.
%\end{equation*} 
Then for any positive integer~$r$ the following holds:  if ${\alpha \in K}$ is a ${2^{r+k}}$th power in $K(i)$, then it is a $2^r$th power in~$K$. In symbols: 
${K(i)^{2^{r+k}}\cap K \subset K^{2^r}}$.    
\end{proposition}

\begin{proof}
%We may assume that~$k$ is the smallest integer with property~\eqref{enotin}; that is, ${\zeta_{2^k}\in K(i)}$. 
Let ${\alpha \in K}$ be a ${2^{r+k}}$th power in $K(i)$. Write ${\alpha=\beta^{2^{r+k}}}$, where ${\beta \in K(i)}$. If ${\beta \in K}$ then there is nothing to prove, so we assume that ${\beta\notin K}$ and we denote by~$\bar\beta$ the conjugate of~$\beta$ over~$K$. Let~$m$ be the smallest integer with the property  ${\beta^{2^m}\in K}$. Then %${(\beta/\bar\beta)^{2^{m-1}}=-1}$, which implies that 
${\beta/\bar\beta}$ is a primitive $2^m$th root of unity. Since ${\beta/\bar\beta\in K(i)}$, and ${\zeta_{2^{k+1}}\notin K(i)}$, we must have   ${m\le k}$; in particular, ${\beta^{2^k}\in K}$. Hence ${\alpha =(\beta^{2^k})^{2^r}\in K^{2^r}}$, as wanted. 
\end{proof}

\section{The Key Lemma}
\label{skey}

The theorem of Chevalley treats the case ${\alpha \in K(\zeta_{p^r})^{p^r}}$. If we want to generalize this, we need  to treat ${\alpha \in K(\zeta_{n})^{p^r}}$ with arbitrary~$n$. This is accomplished with the help of the following key lemma, which is, probably, inspired by Chevalley's Proposition~\ref{pnoti}.  

\begin{lemma}[Smith]
\label{lsm}
Let~$K$ be a field and~$p$ a prime number distinct from $\charac K$. If ${p=2}$ then we assume that ${i\in K}$. Let~$L$ be a finite Galois extension of~$K$, and let~$\ell$ be a non-negative integer such that ${\zeta_{p^{\ell+1}}\notin L}$. Then for any positive integer~$r$  the following holds: if ${\alpha \in K}$ is a ${p^{r+\ell}}$th power in~$L$, then it is a $p^r$th power in~$K$. In symbols: 
${L^{p^{r+\ell}}\cap K \subset K^{p^r}}$.   
\end{lemma}

%Note that Proposition~\ref{pnoti} is a special case of this lemma, corresponding to ${p=2}$ and ${n=4}$. 

\begin{proof} 
%We may assume that~$L$ is a finite extension of~$K$ and argue by induction in ${[L:K]}$. 
Let ${\alpha \in K}$ be a ${p^{r+\ell}}$th power in~$L$. Write ${\alpha=\beta^{p^{r+\ell}}}$, where ${\beta \in L}$. If ${\beta \in K}$ then there is nothing to prove, so we assume that ${\beta\notin K}$. 
Let~$m$ be the smallest integer with the property  ${\beta^{p^m}\in K}$. If ${m\le \ell}$ then ${\beta^{2^\ell}\in K}$, and  ${\alpha =(\beta^{2^\ell})^{2^r}\in K^{2^r}}$, as wanted.

Now assume that ${m>\ell}$. Then ${\gamma:=\beta^{p^m}}$  is a $p^\tho$ power in~$K$. Indeed, in the opposite case polynomial ${X^{p^m}-\gamma}$ would be irreducible over~$K$ by Corollary~\ref{cps}. Since it has a root~$\beta$ in~$L$, which is a Galois extension of~$K$, all its roots must belong to~$L$. In particular, ${\beta\zeta_{p^m}\in L}$, which implies that ${\zeta_{p^m}\in L}$. Since ${m>\ell}$, this contradicts our hypothesis ${\zeta_{p^{\ell+1}}\notin L}$. 

Thus, ${\gamma=\eta^p}$ for some ${\eta\in K}$. Note that ${\eta\ne\beta^{p^{m-1}}}$ by the definition of~$m$. Hence ${\eta/\beta^{p^{m-1}}}$ is a primitive $p^\tho$ root of unity; in particular, ${\zeta_p\in L}$. 

If ${\zeta_p\in K}$ then ${\beta^{p^{m-1}}\in K}$, contradicting the definition of~$m$. Hence~$\zeta_p$ belongs to~$L$, but not to~$K$; in particular, ${[L:K(\zeta_p)]<[L:K]}$. Using induction in the degree ${[L:K]}$, we conclude that ${\alpha\in K(\zeta_p)^{p^r}}$. Theorem~\ref{thchbis} now implies that  ${\alpha\in K^{p^r}}$. 
\end{proof}

\section{Proof of Theorem~\ref{thbass}}
\label{sbass}

Starting from this section, we assume that the field~$K$ is as in Theorem~\ref{thbass}; that is, it is of characteristic~$0$, and its maximal abelian subfield $K_\ab$ is finite over~$\Q$.  Let us introduce some notation. %First of all, we denote ${d:=[K:\Q]}$, the absolute degree of~$K$. 

Let ${p\ge 3}$ be an odd prime number and~$\gerp$ a prime of $K_\ab$  above~$p$. Since $K_\ab$ is a Galois extension of~$\Q$,  the ramification index $e_{\gerp/p}$ depends only on~$p$, and not on the particular choice of~$\gerp$. We denote it by~$\epsilon_p$: % the greatest common divisor of the ramification indices of all $K$-primes above~$p$:
$$
\epsilon_p:=e_{\gerp/p}, \quad   \text{ where $\gerp$ is a prime of~$K_\ab$,\quad $\gerp\mid p$}.
$$
We say that~$p$ is \textit{distinguished} if ${p-1\mid \epsilon_p}$. 

Note that a distinguished prime must ramify in~$K_\ab$. In particular, there exist at most finitely many distinguished primes. 

The ``oddest'' prime~$2$, requires, as often, special consideration. Recall that ${1+i}$ is the $\Q(i)$-prime  above~$2$. We define~$\epsilon_2$ as the  ramification index over ${1+i}$ of a $K_\ab(i)$-prime above ${1+i}$:
$$
\epsilon_2:=e_{\gerp/(1+i)}, \quad \text{where $\gerp$ is a prime of~$K_\ab(i)$,\quad $\gerp\mid 1+i$}.
$$ 
Again,~$\epsilon_2$ is well-defined because ${e_{\gerp/(1+i)}}$ is independent on the particular choice of~$\gerp$.  

%This definition is motivated by the following proposition.

\begin{proposition}
\label{plam}
Let~$\ell$ be a positive integer. 
\begin{enumerate}
\item
Let~$p$ be an odd prime number.  
If ${\zeta_{p}\in K(\zeta_m)}$ for some~$m$ not divisible by~$p$, then~$p$ is distinguished. Moreover, if  ${\zeta_{p^\ell}\in K(\zeta_m)}$, then  ${\ell\le \nu_p(\epsilon_p)+1}$. 

\item
Assume that ${\zeta_{2^\ell}\in K(i,\zeta_m)}$ for some odd integer~$m$. %(In particular, ${\ell\ge 2}$.) 
Then ${\ell\le \nu_2(\epsilon_2)+3}$. 
\end{enumerate}

\end{proposition}

\begin{proof}
To start with, we claim that $K_\ab(\zeta_m)$ is the maximal abelian subfield of ${K(\zeta_m)}$:
\begin{equation}
\label{elabkabz}
K(\zeta_m)_\ab=K_\ab(\zeta_m).
\end{equation}
Indeed, we have 
$$
[K(\zeta_m):K]= [K_\ab(\zeta_m):K_\ab(\zeta_m)\cap K], 
$$
see, for instance, Theorem~1.12 from \cite[Chapter~VI]{La02}. 
By the definition of~$K_\ab$, we have ${K_\ab(\zeta_m)\cap K=K_\ab}$, which gives the equality
$$
[K_\ab(\zeta_m):K_\ab]=[K(\zeta_m):K]. 
$$
Similarly, setting ${L:=K(\zeta_m)}$, we have ${L_\ab\cap K=K_\ab}$, which implies that 
${[L_\ab:K_\ab]=[KL_\ab:K]}$. Since ${KL_\ab=K(\zeta_m)}$, this proves that 
$$
[K_\ab(\zeta_m):K_\ab]=[L_\ab:K_\ab], 
$$
which yields~\eqref{elabkabz} because ${K_\ab(\zeta_m)\subseteq L_\ab}$.

Assume that~$p$ is an odd prime. 
Let~$\gerp$ be a prime of~$K_\ab$ above~$p$ and~$\gerP$ a prime of $K_\ab(\zeta_m)$ above~$\gerp$. Since ${p\nmid m}$ by the hypothesis,~$\gerp$ does not ramify in $K_\ab(\zeta_m)$; in particular, ${e_{\gerp/p}=e_{\gerP/p}}$. 

On the other hand,~$p$ is totally ramified in $\Q(\zeta_{p^\ell})$, the ramification index being ${p^{\ell-1}(p-1)}$. If ${\zeta_{p^\ell}\in K(\zeta_m)}$, then ${\zeta_{p^\ell}\in K_\ab(\zeta_m)}$ by~\eqref{elabkabz}, which implies that
${p^{\ell-1}(p-1)\mid e_{\gerP/p}}$. 
Since ${e_{\gerP/p}=e_{\gerp/p}=\epsilon_p}$, this proves the proposition for the odd~$p$.

The case ${p=2}$  is similar.  If ${\ell\le 3}$ then there is nothing to prove, so let us assume that ${\ell\ge 4}$. Let~$\gerp$ be a prime of~$K_\ab(i)$ above ${1+i}$ and~$\gerP$ a prime of $K_\ab(i,\zeta_m)$ above~$\gerp$. Since~$m$ is odd, we have ${e_{\gerp/(1+i)}=e_{\gerP/(1+i)}}$. 

The prime ${1+i}$ is totally ramified in $\Q(\zeta_{2^\ell})$, the ramification index being ${2^{\ell-3}}$. If ${\zeta_{2^\ell}\in K(i, \zeta_m)}$, then ${\zeta_{2^\ell}\in K_\ab(i, \zeta_m)}$, because 
${K(i, \zeta_m)_\ab=K_\ab(i,\zeta_m)}$, 
which is proved in the same fashion as~\eqref{elabkabz}. 
This implies that
$$
2^{\ell-3}\mid e_{\gerP/(1+i)}=e_{\gerp/(1+i)}=\epsilon_2, 
$$
which  
proves the proposition for ${p=2}$ as well.
\end{proof}

This proposition motivates the following definition. For a prime number~$p$, set
$$
\lambda_p:=
\begin{cases}
\nu_p(\epsilon_p)+1, & \text{if $p$ is distinguished}; \\
\nu_2(\epsilon_2)+3, & \text{if ${p=2}$}; \\
0, & \text{for the other~$p$}. 
\end{cases}
$$
%Theorem~\ref{thbass} is an immediate consequence 

The following statement is the technical heart of the proof of Theorem~\ref{thbass}.

\begin{proposition}
\label{pheart}
Let~$n$ be a positive integer,~$p$ a prime divisor of~$n$, and ${\alpha \in K}$. We set ${r:=\nu_p(n)}$, so that  ${p^r\|n}$. We write ${\lambda=\lambda_p}$, to simplify notation. 

Assume that ${p\ge 3}$, or that ${p=2}$ and ${i\in K}$.  Let~$\alpha$ be a ${p^{\lambda+r}}$th power in $K(\zeta_{np^{\lambda}})$. Then~$\alpha$ is a $p^r$th power in~$K$. In symbols: ${K(\zeta_{np^{\lambda}})^{p^{\lambda+r}}\cap K\subset K^{p^r}}$. 

When ${p=2}$ and ${i\notin K}$,   the same statement holds, but with~$\lambda$ replaced by $2\lambda$;  that is, ${K(\zeta_{n\cdot 4^\lambda})^{2^{2\lambda+r}}\cap K\subset K^{2^r}}$. 
\end{proposition}

\begin{proof}
Write ${n=p^rm}$, so that ${p\nmid m}$. Assume first that ${p\ge 3}$, or that ${p=2}$ and ${i\in K}$. We have ${K(\zeta_{np^\lambda})=K(\zeta_m, \zeta_{p^{\lambda+r}})}$. If ${\alpha \in K(\zeta_{np^\lambda})^{p^{\lambda+r}}}$, then, 
applying Theorem~\ref{thch} with ${\lambda+r}$ as~$r$ and  $K(\zeta_m)$ as~$K$, we obtain ${\alpha \in K(\zeta_m)^{p^{\lambda+r}}}$.

Let~$\ell$ be the biggest integer with the property ${\zeta_{p^\ell}\in K(\zeta_m)}$.
Proposition~\ref{plam} implies that ${\ell\le \lambda}$; in particular,  ${\alpha \in K(\zeta_m)^{p^{\ell+r}}}$. Finally, 
applying Lemma~\ref{lsm} with ${L=K(\zeta_m)}$, we obtain   ${\alpha \in K^{p^{r}}}$.

The case ${p=2}$ and ${i\notin K}$ reduces to the case ${i\in K}$ with the help of Proposition~\ref{pnoti}. Let~$k$ and~$\ell$ be the biggest integers with the properties  
${\zeta_{2^k}\in K(i)}$ and  ${\zeta_{2^\ell}\in K(i,\zeta_m)}$,
respectively. Then ${k\le \ell \le \lambda}$. Arguing as before, we prove the following: if ${\alpha \in K(i,\zeta_{n\cdot 4^\lambda})^{2^{2\lambda+r}}}$,  then ${\alpha \in K(i)^{2^{k+r}}}$. Proposition~\ref{pnoti} implies now that ${\alpha \in K^{2^{r}}}$.
\end{proof}

Now we are ready to complete the proof of Theorem~\ref{thbass}. 
%We claim that the statement of Theorem~\ref{thbass} holds true with 
Set
\begin{equation}
\label{elambda}
\Lambda := 4^{\lambda_2}\prod_{p\ge 3}p^{\lambda_p}. 
\end{equation}
The product is well-defined, because only distinguished primes contribute to it: we have ${\lambda_p=0}$ for the other odd~$p$. %In particular, the product is well-defined, because only finitely many  terms in it are distinct from~$1$. 

Let~$n$ be a positive integer, and let ${\alpha\in K}$ satisfy ${\alpha \in K(\zeta_{n\Lambda})^{n\Lambda}}$. For ${p\mid n}$ we define 
$$
r_p:= \nu_p(n), \qquad N_p:= 
\begin{cases}
n\Lambda/p^{\lambda_p}, &\text{if $p\ge 3$}, \\
n\Lambda/4^{\lambda_2}, &\text{if $p=2$}.
\end{cases}
$$
Then ${p^{r_p}\|N_p}$ and
$$
\alpha \in 
\begin{cases}
K(\zeta_{N_p\cdot p^{\lambda_p}})^{p^{\lambda_p+r_p}},  &\text{if $p\ge 3$}, \\
K(\zeta_{N_2\cdot 4^{\lambda_2}})^{2^{2\lambda_2+r_2}}, &\text{if $p=2$}. 
\end{cases}
$$
Applying Proposition~\ref{pheart} with~$N_p$ as~$n$, we obtain ${\alpha \in K^{p^{r_p}}}$. Theorem~\ref{thbass} now follows, because  
${K^n=\bigcap_{p\mid n}K^{p^{r_p}}}$.

\begin{remark}
\label{relam}
It follows from Proposition~\ref{pheart} that, when ${i\in K}$, one may replace $4^{\lambda_2}$ by $2^{\lambda_2}$  in~\eqref{elambda}. 
\end{remark}

\section{The Chevalley-Bass Number of a Field}
\label{schba}
Let~$K$ be a field as in Theorem~\ref{thbass}; that is, of characteristic~$0$ and with finite degree ${[K_\ab:\Q]}$. Theorem~\ref{thbass} implies that there exist a positive integer~$\Lambda$ such that 
\begin{equation}
\label{echba}
K(\zeta_{\Lambda n})^{\Lambda n}\cap K\subset K^n \qquad (n=1,2,3,\ldots). 
\end{equation} 
Call any such~$\Lambda$ suitable for~$K$. The smallest suitable~$\Lambda$ will be called the \textit{Chevalley-Bass Number} of~$K$ and denoted $\Lambda_K$.

\begin{proposition}
Let~$K$ be as above. 
\begin{enumerate}
\item
\label{igcd}
If~$\Lambda_1$ and~$\Lambda_2$ are suitable for~$K$, then so is $\gcd(\Lambda_1,\Lambda_2)$.

\item
\label{imultip}
If~$\Lambda$ is suitable for~$K$, then so is every  positive integer divisible by~$\Lambda$.  

\item
\label{iall}
A positive integer~$\Lambda$ is suitable for~$K$ if and only if it is divisible by the Chevalley-Bass number~$\Lambda_K$.   
\end{enumerate}
\end{proposition}

\begin{proof}
The proof is very easy. 
For item~\ref{igcd}, let us denote ${\Lambda:=\gcd(\Lambda_1,\Lambda_2)}$ and show that~\eqref{echba} holds.  Since~$\Lambda_1$ is suitable, we  have
$$
K(\zeta_{\Lambda n})^{\Lambda_1 n}\cap K\subset K(\zeta_{\Lambda_1 n})^{\Lambda_1 n}\cap K \subset K^n.  
$$
Similarly, ${K(\zeta_{\Lambda n})^{\Lambda_2 n}\cap K \subset K^n}$. Since  
${
K(\zeta_{\Lambda n})^{\Lambda_1 n}\cap K(\zeta_{\Lambda n})^{\Lambda_2 n}=K(\zeta_{\Lambda n})^{\Lambda n},} 
$
this proves~\eqref{echba}. 

To prove item~\ref{imultip}, assume that~$\Lambda$ is suitable. Then so is its multiple $k\Lambda$, just by applying~\eqref{echba} with~$n$ replaced by $kn$. 

Finally, item~\ref{iall} follows from the two previous items; indeed, item~\ref{igcd} implies that any suitable number is divisible by~$\Lambda_K$, and item~\ref{imultip} implies that every multiple of~$\Lambda_K$ is suitable. 
\end{proof}

\subsection{Estimating the Chevalley-Bass Number}
\label{ssestimchb}
It does not look easy to determine the exact value of the Chevalley-Bass number of a given field~$K$, but it is easy to estimate it. Below we give two such estimates,  one in terms of the degree of~$K_\ab$ and the other in terms of the discriminant. 

%\subsection{Degree} 

\begin{proposition}
\label{prdeg}
Let~$K$ be as above. Denote by~$d$ the degree ${[K_\ab:\Q]}$. 
\begin{enumerate}
\item
\label{igen}
Set 
\begin{align*}
\Delta&:=\prod_{p-1\mid d}p, & \Delta_0 &:= 2^{\nu_2(d)+5}\Delta \prod_{p-1\mid d}p^{\nu_p(d)},\\
\Delta_1&:=2^{\nu_2(d)+5}d\Delta, & \Delta_2&:= 32d^2\Delta. 
\end{align*}
Then the Chevalley-Bass number~$\Lambda_K$ divides  each of the numbers $\Delta_0$, $\Delta_1$ and $\Delta_2$. 

\item
\label{iiin}
If ${i\in K}$ then ${ \Lambda_K\mid 4d\Delta}$. 

\item
\label{iodd}
If~$d$ is odd then $\Lambda_K$ is one of the five numbers ${4,8,16,32,64}$.

\item
\label{iest}
If ${d\ge 3}$ then 
\begin{equation}
\label{elale}
\Lambda_K\le \exp\exp\left(1.7\frac{\log d}{\log\log d}\right). 
\end{equation}
In particular,  ${\Lambda_K\le\exp(d^{o(1)})}$ as ${d\to\infty}$.  
\end{enumerate}
\end{proposition}

The proof requires some preliminary facts collected in the following lemma. 

\begin{lemma}
Let~$n$ be a positive integer. Then the product of all divisors of~$n$ is equal to $n^{\tau(n)/2}$, where $\tau(n)$ is the number of divisors of~$n$:
\begin{equation}
\label{eprodiv}
\prod_{m\mid n}m=n^{\tau(n)/2}. 
\end{equation}
We also have the estimates 
\begin{align}
\label{eutau}
\log\tau(n) &\le 1.07\frac{\log n}{\log \log n} &&(n\ge3), \\
\label{emph}
\prod_{p\le x} \frac{p}{p-1}&\le 2\log (x-1) %e^\gamma\left(\log x  +\frac{1}{\log x}\right) 
&&(x\ge 25). 
\end{align}
%where 
\end{lemma}

\begin{proof}
We have 
$$
\left(\prod_{m\mid n} m\right)^2=\prod_{m\mid n} m\prod_{m\mid n}  \frac nm= n^{\tau(n)}, 
$$
which proves~\eqref{eprodiv}. For~\eqref{eutau} see \cite[Théorème~1]{NR83}. Finally,~\eqref{emph} follows easily from the estimate  
$$
\prod_{p\le x} \frac{p}{p-1}\le 1.8\left(\log x  +\frac{1}{\log x}\right), 
$$
which holds for all ${x>1}$, see \cite[Corollary~1 of Theorem~8]{RS62}.%;  here ${\gamma=0.57721\ldots}$ is Euler's constant. 
\end{proof}

\begin{proof}[Proof of Proposition~\ref{prdeg}]
Since  ${\epsilon_p \mid d}$  for every~$p$, we have ${p-1\mid d}$ for a distinguished~$p$. %(In particular, if there is at least one distinguished prime, then~$d$ is even.) 
Hence the right-hand side of~\eqref{elambda} divides 
$$
2^{2\nu_2(d)+6}\prod_{\genfrac{}{}{0pt}{}{p-1\mid d}{p\ge 3}}p^{1+\nu_p(d)}=  \Delta_0. 
$$
Clearly, ${\Delta_0\mid\Delta_1\mid \Delta_2}$. This proves item~\ref{igen}. 

If ${i\in K}$ then ${2\nu_2(d)+6}$ above can be replaced by ${\nu_2(d)+3}$, see Remark~\ref{relam}. Hence ${\nu_2(d)+5}$ in the definition of~$\Delta_1$ can be replaced by~$2$. This proves item~\ref{iiin}.

If~$d$ is odd then ${\Delta_0=64}$.  Also, ${i\notin K}$ when~$d$ is odd,  and the already mentioned example ${-4=(1+i)^4}$ implies that ${\Lambda_K\ne 1,2}$. This proves item~\ref{iodd}.

We are left with the estimate~\eqref{elale}. It is clear for odd~$d$, because the right-hand side of~\eqref{elale} exceeds ${\exp\exp(1.7e)> 64}$. 
Running a simple \textsf{PARI} script, we check that~$\Delta_1$ does not exceed the right-hand side of~\eqref{elale} for even ${d\le 10^7}$. (The total computational time was less than 5 minutes on an ordinary laptop.) Hence we have to prove that 
\begin{equation}
\label{ejjj}
\log\Delta_2 \le \exp\left(1.7\frac{\log d}{\log\log d}\right)
\end{equation}
for ${d\ge 10^7}$. 
Using~\eqref{eprodiv} and~\eqref{eutau}, we obtain
$$
\Delta\le \prod_{m\mid d}m\prod_{p\le d+1}\frac{p}{p-1}\le d^{\tau(d)/2}\cdot 2\log d. 
$$ 
It follows that
$$
\log\Delta_2 \le \frac12\tau(d)\log d+2\log d+\log\log d+\log 64\le 2\tau(d)\log d, 
$$
where for the last inequality we used ${\tau(d)\ge 2}$ and ${d\ge 10^7}$. Next, using~\eqref{eutau}, we obtain 
$$
\log\Delta_2 \le \exp\left(1.07\frac{\log d}{\log \log d}+\log(2\log d)\right). 
$$ 
When ${d\ge 10^7}$ we have 
$
{\log(2\log d) \le 0.6\log d/\log \log d} 
$.
This proves~\eqref{ejjj}. 
\end{proof}

Since the distinguished primes  ramify in~$K$, it is natural to expect a simple expression for a suitable~$\Lambda$ in terms of the absolute discriminant ${\DD:=|\DD_{K_\ab}|}$. 

\begin{proposition}
The Chevalley-Bass number $\Lambda_K$ divides ${64\DD}$. If ${i\in K}$ then  ${\Lambda_K\mid 2\DD}$. 
\end{proposition}

\begin{proof}
By Remark~\ref{relam}, we need to show that 
\begin{equation}
\label{ewhattoprove}
\lambda_p \le 
\begin{cases}
\nu_p(\DD), & \text{if $p$ is distinguished}, \\
\nu_2(\DD)+1, &\text{if $p=2$ and $i\in K$},\\
 \nu_2(\DD)/2+3, &\text{if $p=2$ and ${i\notin K}$}. 
\end{cases}
\end{equation}
We will be using the following well-known fact: if~$L$ is number field and~$p$ a prime number, then 
\begin{equation*}
%\label{enupdis}
\nu_p(\DD_L)\ge \sum_{\gerp\mid p}(e_{\gerp/p}-1+\delta_\gerp)f_{\gerp/p}, \qquad \text{where}\quad
\delta_\gerp :=
\begin{cases}
0, &\text{if ${p\nmid e_{\gerp/p}}$}, \\
1, &\text{if ${p\mid e_{\gerp/p}}$}, 
\end{cases}
\end{equation*}
the sum being over the $L$-primes above~$p$. 
When ${L=K_\ab}$, we obviously have ${e_{\gerp/p}=\epsilon_p}$ for ${p\ge 3}$. For ${p=2}$,  denote by~$\gerP$ the prime of $K(i)$ above the $K$-prime~$\gerp$. Note that that ${\epsilon_2=e_{\gerP/(1+i)}}$. Then  
$$
e_{\gerP/2}=e_{(1+i)/2}e_{\gerP/(1+i)}=e_{\gerp/2}e_{\gerP/\gerp}. 
$$
Since  
$$
e_{(1+i)/2}=2, \qquad e_{\gerP/(1+i)}=\epsilon_2, \qquad e_{\gerP/\gerp}\in \{1,2\}, 
$$
this shows that ${e_{\gerp/2}\in \{\epsilon_2,2\epsilon_2\}}$, and ${e_{\gerp/2}=2\epsilon_2 }$ when ${i\in K}$. 

Thus, in any case we have  
\begin{equation*}
%\label{enupdis}
\nu_p(\DD)\ge \epsilon_p-1+\delta_p, \qquad \text{where}\quad 
\delta_p :=
\begin{cases}
0, &\text{if ${p\nmid \epsilon_p}$}, \\
1, &\text{if ${p\mid \epsilon_p}$}, 
\end{cases}
\end{equation*}
and in the special case ${p=2}$, ${i\in K}$ we have 
${\nu_2(\DD)\ge 2\epsilon_2}$. Hence, to establish~\eqref{ewhattoprove}, we need to show that
\begin{equation}
\label{ewhattoshow}
\nu_p(\epsilon_p) 
\le 
\begin{cases}
\epsilon_p-2+\delta_p, & \text{if $p$ is distinguished}, \\
2\epsilon_2-2, &\text{if $p=2$ and $i\in K$},\\
(\epsilon_2-1+\delta_2)/2, &\text{if $p=2$ and ${i\notin K}$}. 
\end{cases}
\end{equation}
When ${p\nmid \epsilon_p}$,  the inequalities in~\eqref{ewhattoshow} are true (note that ${\epsilon_p\ge p-1\ge 2}$ when~$p$ is distinguished). Hence we may assume that ${p\mid \epsilon_p}$; in particular, ${\delta_p=1}$ and ${\epsilon_p\ge p(p-1)}$. 
Since ${\nu_p(\epsilon_p)\le \log\epsilon_p/\log p}$ and ${\epsilon_p-1\ge \epsilon_p/2}$, condition~\eqref{ewhattoshow} follows from  
%\begin{equation*}
%\label{epge}
${\log\epsilon_p/\log p\le\epsilon_p/2}$.   
%\end{equation*}
This  holds  when 
${\epsilon_p=p=2}$.  In the remaining cases we have ${\epsilon_p\ge 4}$; to treat these remaining cases, just note that 
the function ${x\mapsto \log x/\log 2-x/2}$
is decreasing for ${x\ge 4}$, and vanishes at ${x=4}$. 
\end{proof}

\subsection{The Chevalley-Bass Number of the Splitting Field}

For some applications, it is of interest to estimate the Chevalley-Bass number of the splitting field of a rational polynomial. 
\begin{proposition}
\label{prspl}
Let ${f(T)\in \Q[T]}$ be a polynomial of degree ${m\ge 2}$, and let~$K$ be the splitting field of~$f$. Then 
$
{\Lambda_K\le \exp\exp\left({m}/{\log m} \right)} 
$. 
\end{proposition}

\begin{proof}
The Galois group ${G:=\Gal(K/\Q)}$ is a subgroup of the symmetric group $\calS_m$, and ${d:=[K_\ab:\Q]}$ is the order of the maximal abelian quotient of~$G$; that is, 
${d=\#G/[G,G]}$. 
It is known that ${d\le 3^{m/3}}$, see Kovács~\& Praeger \cite[Corollary on page 284]{KP89},  or Aschbacher~\& Guralnick \cite[Theorem~2]{AG89}. 

A quick calculation with \textsf{PARI} implies that for ${m\le 15}$ and ${d\le 3^{m/3}}$ we have ${\log\Delta_0\le \exp(m/\log m)}$, where~$\Delta_0$ is from Proposition~\ref{prdeg}. Now assume that ${m\ge 16}$. Estimate~\eqref{elale} implies that, when ${d\le 3^{m/3}}$, we have 
\begin{align*}
\log\Lambda_K&\le \exp\left(1.7\frac{\log 3}{3} \frac{m}{\log m -\log(3/\log3)}\right)\\
&\le \exp\left(1.7\frac{\log 3}{3} \frac{\log16}{\log 16 -\log(3/\log3)} \frac{m}{\log m}\right)\\
&\le \exp\left(0.98 \frac{m}{\log m}\right), 
\end{align*}
which is even better than wanted. 
\end{proof}

\subsection{Open Questions}
\label{ssopen}

In spite of the results of Subsection~\ref{ssestimchb}, the Chevalley-Bass number of a field remains a mysterious quantity. In particular, we do not know its exact value for a single field. Here are some questions that we would like to have answered.

\begin{enumerate}
\item
What is the Chevalley-Bass number of~$\Q$? and of $\Q(i)$?

\item
Is finding the Chevalley-Bass number of a given number field decidable? Because of the results of Subsection~\ref{ssestimchb}, this reduces to the following formally easier problem: given a number field~$K$ and a positive integer~$\Lambda$,  decide whether or not~$\Lambda$ is suitable for~$K$.  

\item
Does it exist a field~$K$ with ${\Lambda_K=1}$? (Perhaps, $\Q(i)$ is such field.)

\item
Does it exist a field~$K$ with~$\Lambda_K$ divisible by an odd prime? 

\item
Can the estimate of Proposition~\ref{prspl} be refined? We believe that an estimate of the shape $\exp\exp(O((m/\log m)^{1/2}))$ must hold. 
\end{enumerate}

We hope to see some of this questions answered in not too distant future.

%\section{The Number Field Case}\label{snf}

{\footnotesize

\bibliographystyle{amsplain}
\bibliography{chevbass}

}
\end{document}